\RequirePackage{fix-cm}
\documentclass[smallcondensed,numbook,envcountsame, reqno]{svjour3}     
\smartqed
\usepackage{graphicx}
 \usepackage{mathptmx}      

\usepackage[leqno]{amsmath}
\usepackage{amsfonts}
\usepackage{amssymb}
\usepackage[all]{xy}

\newcommand{\rr}{\mathbb R}

\newcommand{\LL}{\mathcal L}
\newcommand{\pp}{\mathbb P}
\newcommand{\ppx}{\mathbb P^x}
\newcommand{\ee}{\mathbb E}
\newcommand{\eex}{\mathbb E^x}
\newcommand{\eey}{\mathbb E^y}
\newcommand{\eez}{\mathbb E^z}
\newcommand{\calB}{\mathcal B}

\spdefaulttheorem{assumption}{Assumption}{\bf}{\it}
\spdefaulttheorem{prop}{Proposition}{\bf}{\it}

\numberwithin{equation}{section}

 \journalname{}
\begin{document}

\title{Occupation Times for Stable-like Processes
}


\author{Brian M. Whitehead}         

\authorrunning{B.M. Whitehead} 

\institute{B. Whitehead \at
              Department of Mathematics\\
	    University of Connecticut\\
	    Storrs, CT 06269, U.S.A. \\
              Tel.: 860-486-1287\\
              \email{brian.whitehead@uconn.edu}       }    

\date{Received: date / Accepted: date}

\maketitle

\begin{abstract}
We consider a class of pure jump Markov processes in $\rr^d$ whose jump kernels are comparable to those of symmetric stable processes.  We prove a support theorem, a lower bound on the occupation times of sets, and show that we can approximate resolvents using smooth functions.

\end{abstract}

\section{Introduction}\label{intro}

In this paper, we will consider pure jump Markov processes in $\rr^d$ whose jump structures are comparable to those of a symmetric stable process of index $\alpha$.  In recent years there has been increased interest in jump processes, since they appear to better model certain physical and financial phenomena.  Here, we prove that a support theorem holds, that there is a lower bound on occupation times of sets, and that we can approximate resolvents using smooth functions.

We will consider processes associated to the operator
\begin{equation}
\LL f(x) = \int _{\rr^d - \{0\}} [f(x+h)-f(x) - \nabla f(x) \cdot h1_{(|h| \leq 1)}] n(x,h) dh,
\end{equation}
which is a non-local operator, in the sense that the behavior of $\LL f$ at a point $x$ depends on the values of $f$ at distant points.  
Non-local operators have been studied frequently in the last several years.  In \cite{bbg}, a Liouville property was shown, and in 2002, Bass and Levin \cite{bl} gave the first proof of a Harnack inequality for such an operator.  This result was then generalized in \cite{sv} and \cite{bk}.  

We use the notion of the martingale problem to associate a process $X_t$ to $\LL$.  Then $n(x,h)$ gives us the intensity of the number of jumps from $x$ to $x+h$.  If we set $n(x,h) = \frac{c_1}{|h|^{d + \alpha}}$, $X_t$ would be a symmetric stable process of index $\alpha$.  Here, we allow $n(x,h)$ to depend on location in $\rr^d$ as well as the size of the jumps. 

 A support theorem is a result which states that there will be some positive probability that the processes we are considering will not stray too far from the image of any given continuous map $\varphi \colon [0,t] \to \rr^d$.  That is, if we fix $\varepsilon > 0$, and let $\varphi(0) = x_0$, then there exists $c_1 > 0$ depending on $\varphi$, $\varepsilon$, and $t_0$ such that
$$\pp^{x_0}\left(\displaystyle \sup_{s\leq t_0} |X_s - \varphi(s)| < \varepsilon \right) > c_1.$$
    Support theorems proven in other contexts have been useful tools in further proofs.  Bass and Chen \cite{bc2} showed that a support theorem holds for a different class of jump processes which are only allowed to jump in finitely many different directions, while the processes we are considering can jump in any direction.

Bass and Levin \cite{bl} showed that a Harnack inequality holds for the processes we are considering.  As part of this proof, they demonstrated that these processes will hit sets of positive Lebesgue measure with positive probability.  With the additional assumption of some continuity in the jump kernel, we are able to extend this result in Theorem \ref{t43}, by showing that such processes will be expected to spend a positive amount of time in sets of positive Lebesgue measure.  In particular, we define the occupation time of a set $B$ to be 
\begin{equation}
\eex \int_0^\tau 1_B(X_s)ds,
\end{equation} 
where $\tau$ is the first time we leave some ball in $\rr^d$ containing the set $B$ and $|B|$ is the Lebesgue measure of $B$.  We show that there exists a nondecreasing function $\varphi : (0,1) \to (0,1)$ such that if $x \in Q(0, 1/2)$ and $B \subseteq Q(0,1)$, then
$$ \eex \int_0^{\tau_{Q(0,1)}}1_B(X_s) ds \geq \varphi(|B|). $$ 
Here $Q(0,r)$ denotes the cube centered at 0 with side length $r$.

Finally, we will use these results in order to show that we can approximate resolvents using smooth functions, a result similar to one known for the continuous processes associated to nondivergence form operators \cite{bass}.  We suppose that $n^\varepsilon(x,h)$ is in $C^\infty$ with respect to $x$ for any fixed $h$, and we define
\begin{equation}
\LL^\varepsilon f(x) = \int _{\rr^d - \{0\}} [f(x+h)-f(x) - \nabla f(x) \cdot h1_{(|h| \leq 1)}] n^\varepsilon(x,h) \thinspace dh. 
\end{equation}
We will prove that 
there exist $n^\varepsilon(x,h) \in C^\infty$ such that if $\ppx_\varepsilon$ is the solution to the martingale problem for $\LL^\varepsilon$ started at $x$ and 
\begin{equation}
S_\lambda^\varepsilon h(x) = \eex_\varepsilon\int_0^\infty e^{-\lambda t}h(X_t)\thinspace dt
\end{equation}
for $h$ bounded, then
\begin{equation}
S_\lambda^\varepsilon f  \to \ee \int_0^\infty e^{-\lambda t} f(X_t) \thinspace dt
\end{equation}
whenever $f$ is continuous.
We also will show that the $n^\varepsilon(x,h)$ will converge to $n(x,h)$ almost everywhere.

Section 2 contains some preliminaries and states some useful results from \cite{bl}, section 3 contains the proof of the support theorem, section 4 is the proof of the lower bound on occupation times, and in section 5 we consider the approximation of resolvents by smooth functions.

The material in this paper is part of a Ph.D. dissertation under the advisement of Dr. Richard F. Bass.

\section{Preliminaries}\label{prelim}

We define \begin{equation}\label{scriptl}
\LL f(x) = \int _{\rr^d - \{0\}} [f(x+h)-f(x) - \nabla f(x) \cdot h1_{(|h| \leq 1)}] n(x,h) dh,
\end{equation}
where $f \in C^2$ such that $f$ and its first and second partial derivatives are bounded.
We will assume throughout this paper that $(\pp^x,X_t)$ is a strong Markov process with state space $\rr^d$ which has the property that for every $x$ the probability measure $\ppx$ is a solution to the martingale problem for $\LL$ started at $x$, that is, that

\noindent(a) $\pp(X_0 = x) = 1$ 

\noindent(b) for each $f \in C^2$ which is bounded with bounded first and second partial derivatives, 
$$f(X_t) - f(X_0) - \int_0^t \LL f(X_s) \thinspace ds$$
is a $\pp$-martingale.
  
We also make the following assumption, which is identical to the one made in \cite{bl}.

\begin{assumption}\label{a11}
(a) For all $x$ and $h$ we have $n(x,-h) = n(x,h).$
\newline
(b) There exist constants $\kappa  \in (0,1)$ and $\alpha \in (0,2)$ such that for all $x$ and $h$ we have 
\begin{equation} \label{key}
\frac{\kappa}{|h|^{d + \alpha}} \leq n(x,h) \leq \frac{\kappa^{-1}}{|h|^{d + \alpha}}.
\end{equation}
\end{assumption}
We will assume throughout this paper that such a process is given.  Existence is known  however, due to \cite{ba} and \cite{k}, provided an additional smoothness assumption on $n(x,h)$ is given.

Throughout this paper, we denote by $B(x,r)$ the ball of radius $r$ centered at $x$, and by $Q(x,r)$ the cube of side length $r$ centered at $x$.  $|A|$ will denote the Lebesgue measure of $A$.  We denote the hitting and exit times of set $A$ respectively, by $$T_A = \inf\{t>0:X_t \in A \},\qquad \tau_A = \inf\{t>0:X_t \notin A \}.$$
We write $X_{t-} = \lim_{s \uparrow t} X_s$, and $\Delta X_t = X_t - X_{t-}$.

The letter $c$ with subscripts will denote various positive constants with unimportant values.  These  constants will usually depend on $\alpha$, $\kappa$,  and $d$ along with other dependences which will be explicitly mentioned in our results.

Many results regarding processes satisfying Assumption \ref{a11} can be found in \cite{bl}.  As tools for proving a Harnack inequality, the authors showed that a scaling property holds, and proved some other useful results, which we list below and will reference throughout this paper.

\begin{lemma} \label{l24}
  Let $\varepsilon > 0$.  There exists $c_1$ depending only on $\varepsilon$ such that if $x \in \rr^d$ and $r > 0$, then  $$ \inf_{z \in B(x,(1-\epsilon)r)} \ee^z \tau_{B(x,r)} \geq c_1r^\alpha .$$

\end{lemma}

\begin{lemma} \label{l25} 
There exists $c_1$ such that $\sup _z \ee^z \tau_{B(x,r)} \leq c_1 r^\alpha.$
\end{lemma}

\begin{prop}\label{p26}
Suppose  $A  \subseteq B(x,1)$.  There exists $c_1$ not depending on $x$ or $A$ such that
$$ \pp^y(T_A <\tau_{B(x,3)}) \geq c_1|A|, \qquad y \in B(x,2). $$
\end{prop}

\section{Support Theorem}\label{supp}
In this section, we will prove a support theorem for $X$.  This proof is similar to the one given by Bass and Chen \cite{bc2}.  However, our proof will require some different techniques, since the processes we are considering are allowed to jump in any direction, while the processes considered in \cite{bc2} can only jump in finitely many different directions.  We begin by proving some lemmas.

\begin{lemma}\label{l21}
Let $x_0, x_1 \in \rr^d$, $y = x_1 - x_0$, and $\gamma > 0$.  There exists $t_0 > 0$, such that for all $t \leq t_0$ there is a positive constant $c_1$, depending only on $\gamma$, $|y|$, and $t$,  such that 
\begin{align*} 
\pp^{x_0}(&\textit{there exists a stopping time } T \leq t \textit{ such that } \\
 &\displaystyle \sup_{s<T} |X_s - x_0| < \gamma \textit{ and } 
    \sup_{T \leq s \leq t_0} |X_s - x_1| < \gamma) \geq c_1.
\end{align*}
\end{lemma}

\begin{proof}
Let $\beta = \frac{|y|}{2} \wedge 1$, and $\delta = \frac{\gamma}{3} \wedge  \frac{|y|}{6}$.  We define a new operator
\begin{equation}
\LL_\beta f(x) = \int _{|h| < \beta} [f(x+h)-f(x) - \nabla f(x) \cdot h] n(x,h) dh,
\end{equation}
and we let $\overline X$ be the strong Markov process associated to $\LL_\beta$, so that $\overline X$ has no jumps having size larger than $\beta$.  We consider the function $f(x) = |x - x_0|^2$, and we let $\tau = \tau_{B(x_0, \delta)}$.  Since $\pp^{x_0}$ solves the martingale problem for $\LL_\beta$ started at $x_0$, 
$$ f(\overline X_t) - f(x_0) - \int_0^t \LL_\beta f(\overline X_s) \thinspace ds $$
is a martingale.  Therefore, by applying optional stopping, we obtain that
\begin{equation}
\ee^{x_0} f(\overline X_{t \wedge \tau}) - f(x_0) =\ee^{x_0} \int_0^{t \wedge \tau}\LL_\beta f(\overline X_s) \thinspace ds.
\end{equation}
Since $f(x_0) = 0$, we obtain
\begin{equation}\label{e23}
\ee^{x_0}f(\overline X_{t \wedge \tau}) = \ee^{x_0} \int_0^{t \wedge \tau}\LL_\beta f(\overline X_s) \thinspace ds.
\end{equation}
We further observe that
\begin{equation}\label{e24}
\delta^2 \pp^{x_0}(\tau \leq t) \leq \ee^{x_0}f(\overline X_{t \wedge \tau}),
\end{equation}
since $f(x) \geq \delta^2$ outside of $B(x_0, \delta)$.  On the other hand, we have that 
\begin{equation}\label{e25}
\ee^{x_0} \int_0^{t \wedge \tau}\LL_\beta f(\overline X_s) \thinspace ds \leq t \sup_{s \leq t} \LL_\beta f(\overline X_s).
\end{equation}

Now for any $x \in \rr^d$, 
\begin{align*}
\LL_\beta f(x) &= \int_{|h| < \beta} \displaystyle \left[|x + h -x_0|^2 - |x - x_0|^2 - \nabla f(x) \cdot h\right]n(x,h)dh \\
&\leq \int_{|h| < \beta} \left[|x  -x_0|^2 + 2|x-x_0|\cdot|h| + |h|^2 - |x - x_0|^2 - 2|x-x_0|\cdot|h|\right]\frac{\kappa^{-1}}{|h|^{d+\alpha}}dh \\
&= \kappa^{-1} \int_{|h| < \beta} |h|^{2 - (d+\alpha)} = c_2\beta^\alpha
\end{align*}

Therefore, by combining this with \eqref{e23}, \eqref{e24}, and \eqref{e25}, we have that
\begin{equation}\label{e26}
\pp^{x_0}(\tau \leq t) \leq t\frac{c_2\beta^\alpha}{\delta^2}.
\end{equation}
Let $t_0 = \frac{\delta^2}{2c_2\beta^\alpha}$. Let 
$$ E = \displaystyle \left\{ \sup_{s \leq t} |\overline X_s - x_0| \leq \delta \right \}. $$  Thus if $t \leq t_0$, we will have that $\pp^{x_0}(E) \leq 1/2$.

We now will use a construction of Meyer to add some large jumps to the process $\overline X_t$, in order to create a process $X_t$ which will be associated to our operator $\LL$.  A reference for this process is Remark 3.4 of \cite{bbck}.  Let $t \leq t_0$ be fixed, let $U_1$ and $U_2$ be the times of the first two jumps we add to $\overline X$, and define 
$$D =  \{U_1 \leq t < U_2, \Delta X_{U_1} \in B(y,\delta) \}.$$

Let $S_1$ and $S_2$ be independent exponential random variables of parameter 1, which also independent of $\overline X$.    We note that for any $x \in \rr^d$, 
\begin{equation}\label{e27}
c_3 = \int_{|h| \geq \beta} \frac{\kappa}{|h|^{d+\alpha}} \leq \int_{|h| \geq \beta}n(x,h)dh \leq \int_{|h| \geq \beta} \frac{\kappa^{-1}}{|h|^{d+\alpha}} = c_4. 
\end{equation}
We define
$$ F = \{ S_1 \in (0, c_3t], S_2\in [c_4t, \infty) \}. $$
Now, 
\begin{equation*}
\ppx(S_1 \leq c_3t) = 1 - e^{-c_3t},
\end{equation*} and
\begin{equation*}
\ppx(S_2 \geq c_4t) = e^{-c_4t},
\end{equation*}
so by the independence of $S_1$ and $S_2$, $\ppx(F)  \geq c_5$.  Furthermore, the event $F$ was chosen to be independent of $E$, so we have
$$ \pp^{x_0}(F \cap E) =  \pp^{x_0}(F)\ \pp^{x_0}(E) \geq c_5/2.$$

We define
$$C_t = \int_0^t\int_{|h|\geq \beta} n(\overline X_s, h)dh \thinspace ds.$$
Per Meyer's construction, we will introduce an additional jump to $\overline X$ at the first time $U_1$ such that $C_{U_1}$ exceeds $S_1$, restart the process, and then introduce a second jump when $C_{U_2}$ exceeds $S_2$.  It follows from \eqref{e27} that if $F$ holds, we will add  exactly one jump to $\overline X$ before time $t$, so that if $G = \{U_1 \leq t < U_2\}$,
$$ \pp^{x_0}( G \cap E ) \geq  \pp^{x_0}(F \cap E) \geq c_5/2.$$

Suppose now that G holds.  
The location of the jump at time $U_1$ of size larger than $\beta$ will be determined by the distribution
$$q(x,dz) = \displaystyle \frac{n(x,z-x)}{\int_{|h| \geq \beta} n(x,h)dh}dz$$
where $|x - z| \geq \beta$.  Therefore, if $B = B(y, \delta)$, 
\begin{align*}
\pp^{x}(\Delta X_{U_1} \in B) &= \int_B q(X_{U_1-},X_{U_1-} + dz)  \\
&\geq \int_B \displaystyle \frac{n(X_{U_1-}, z)}{c_4}dz  \\
&\geq \frac{|B|\kappa^{-1}}{c_4|z|^{d+\alpha}} \geq \frac{|B|\kappa^{-1}}{c_4[(3/2)|y|]^{d+\alpha}} = c_6, 
\end{align*}
a constant which depends only on $\gamma$ and $|y|$.  This bound does not depend on $\overline X_{U_1}$, so we have that $ \pp^{x_0}(D \cap E) \geq c_5c_6/2$.

We now note that on $D \cap E$, 
$$\displaystyle \sup_{s<U_1} |X_s - x_0| < \delta < \gamma,$$ 
and
$$\displaystyle \sup_{U_1 \leq s \leq t_0} |X_s - x_1| \leq \delta + \delta + \delta <  \gamma,$$ 
so $U_1$ is our desired stopping time.
\qed
\end{proof}

\begin{lemma}\label{l22}
Let $t_1 > 0$, $\varepsilon > 0$, $r \in (0, \varepsilon/4)$, and $\gamma > 0$.  Let $\psi \colon [0, t_1] \to \rr^d$ be a line segment of length $r$ starting at $x_0$.  Then there exists $c_1 > 0$ that depends only on $t_1$, $\varepsilon$, and $\gamma$ such that
$$\pp^{x_0}\displaystyle \left(\sup_{s\leq t_1} |X_s - \psi(s)| < \varepsilon\textit{ and }|X_{t_1} - \psi(t_1)| < \gamma \right) \geq c_1.$$
\end{lemma}

\begin{proof}

Let $D_1$ be the event that there is a stopping time $T < t_0$ such that $|X_s - x_0| < \gamma \wedge\varepsilon/4$ for $ s < T$ and $|X_s - \psi(t_0)| < \gamma \wedge \varepsilon/4$ for $s \in [T, t_0]$. 
 By Lemma \ref{l21}, there exists $t_0 = t_1/n$ for some $n$ and $c_2 > 0$ such that $\pp^{x_0}(D_1) \geq c_2$.  We now note that on $D_1$, by definition, $|X_{t_0} - \psi(t_0)| < \gamma$.   We now show that on $D_1$,  $|X_s - \psi(s)| < \varepsilon$ for every $s \in [0, t_0]$.

If $s < T$, since $r < \varepsilon/4$, we have that
\begin{equation*}
|X_s - \psi(s)| \leq |X_s -x_0| + |x_0 - \psi(s)| < \frac{\varepsilon}{4} + \frac{\varepsilon}{4} < \frac{\varepsilon}{2}. 
\end{equation*}
Similarly, if $T \leq s \leq t_0$, we obtain that
\begin{equation*}
|X_s - \psi(s)| \leq |X_s -\psi(t_0)| + |\psi(t_0) - \psi(s)| < \frac{\varepsilon}{4} + \frac{\varepsilon}{4} < \frac{\varepsilon}{2}. 
\end{equation*}

If $t_0 = t_1$, we are done.  If not, then let $D_k$ be the event that there is a stopping time $(k-1)t_0< T < kt_0$ such that $|X_s - X_{(k-1)t_0}| < \gamma \wedge\varepsilon/4$ for $ s \in [(k-1)t_0, T]$ and $|X_s - \psi(kt_0)| < \gamma \wedge \varepsilon/4$ for $s \in [T, kt_0]$.  By Lemma \ref{l21}, $\pp^{x_0}(D_k) \geq c_3$.  

Now on $D_k$, we observe that when $ s \in [(k-1)t_0, T]$, then
\begin{align*}
|X_s - \psi(s)| &\leq |X_s -X_{(k-1)t_0}| + |X_{(k-1)t_0} - \psi((k-1)t_0) |+ |\psi((k-1)t_0) - \psi(s)|\\
&< \frac{\varepsilon}{4} + \frac{\varepsilon}{4} + \frac{\varepsilon}{4} < {\varepsilon}, 
\end{align*}
and similarly, when $ s \in [T, kt_0]$, $|X_s - \psi(s)| \leq \varepsilon$.  Furthermore, $|X_{kt_0} - \psi(kt_0)| < \gamma$.

Therefore, this lemma follows after applying the Markov property $n$ times.
\qed
\end{proof}

\begin{theorem}\label{support}
Let $\varphi : [0,t_0] \to \rr^d$ be continuous with $\varphi(0)= x_0$.  Let $\varepsilon > 0$.  There exists $c_1 > 0$ depending on $\varphi$, $\varepsilon$, and $t_0$ such that
$$\pp^{x_0}\left(\displaystyle \sup_{s\leq t_0} |X_s - \varphi(s)| < \varepsilon \right) > c_1.$$
\end{theorem}

\begin{proof}
We may approximate $\varphi$ to within $\varepsilon/2$ by a polygonal path, so by changing $\varepsilon$ to $\varepsilon/2$, we may assume that $\varphi$ is polygonal, without loss of generality.  We now choose $n$ large and subdivide the interval $[0,t_0]$ into $n$ subintervals so that for every $0\leq k\leq n-1$, so that over each subinterval $[kt_0/n, (k+1)t_0/n]$  the image of $\varphi$  is a line segment whose length is smaller than $\varepsilon/4$.  By Lemma \ref{l22}, there exists $c_2 >0$,  such that on each time interval  $[kt_0/n, (k+1)t_0/n]$, $$\pp^{x_0}\displaystyle \left(\sup_{kt_0/n\leq s\leq (k+1)t_0/n} |X_s - \psi(s)| < \frac{\varepsilon}{2}\textit{ and }|X_{(k+1)t_0/n} - \psi((k+1)t_0/n)| < \frac{\varepsilon}{4\sqrt{d}} \right) \geq c_2.$$
Now by applying the Markov property $n$ times, we obtain our support theorem.
\qed
\end{proof}

\section{Occupation Times}\label{Occupation}
In this section, we show a lower bound on the occupation time of a set. For the remainder of this paper, we will require some continuity in $x$ of $n(x,h)$.  We take $\eta > 0$, and set 
\begin{equation} \label{psi}
 \psi_\eta(r) = (1 + \log^+(1/r))^{1 + \eta}, \qquad r>0,
\end{equation}
and we take $\bar{n}(x,h) = n(x,h)\psi_\eta(|h|)$. 

The following assumption, which we now combine with Assumption \ref{a11}, will require more continuity in $x$ the smaller that $h$ is.

\begin{assumption} \label{a22}
There exists $\eta > 0$ such that for every $y \in \rr^d$ and every $b > 0$  $$ \lim_{x \to y} \sup_{|h| \leq b} |\bar{n}(x,h) - \bar{n}(y,h)| = 0.$$

\end{assumption}

Before we discuss occupation times, we will need some facts regarding resolvents.
We fix $x_0 \in \rr^d$, and define 
\begin{equation} \label{l0}
 \LL_0 f(x) = \int [f(x+h) - f(x) - \nabla f(x) \cdot h1_{(|h| \leq 1)}] n(x_0,h) dh,
\end{equation}
when $f \in C^2_b$.  Let $R_\lambda$ be the resolvent for the L\'{e}vy process $X_t$ whose infinitesimal generator is $\LL_0$.

Let $\eta > 0$ and let $\psi_\eta$ be as in \eqref{psi}.  We make an additional temporary assumption here.  

\begin{assumption} \label{ouch}
There exists $\zeta$ such that 
$$ | n(x,h) - n(x_0,h) | \leq \frac{\zeta}{\psi_\eta(|h|)|h|^{d+\alpha}}, \qquad x \in \rr^d, \quad |h| \leq 1. $$
\end{assumption}

We adopt the notation that $\calB = \LL - \LL_0.$

\begin{lemma} \label{l32}
Suppose that $p > \max\{1,d/\alpha\}$.  Then
\begin{equation} \label{bdd}
\| R_\lambda f(x) \|_\infty \leq c_1 \|f \|_p. 
\end{equation}
\end{lemma}

\begin{proof}
We have that 
$$ \displaystyle | R_\lambda | = \left | \int_0^\infty e^{-\lambda t}P_tf(x) dx\right|,$$
where $P_t f(x) = \eex \int_0^\infty f(X_s) ds$ is the transition semigroup of $X_t$.
Let $p(t,x-y)$ be the transition density function of $X$ with respect to Lebesgue measure on $\rr^d$, and let $q$ be the conjugate exponent to $p$, so that $\frac{1}{p} + \frac{1}{q} = 1$.
 
We have then, by H\"{o}lder's inequality,
$$ |P_t f(x)| = \displaystyle \left | \int_{\rr^d} p(t,x-y)f(y)dy \right | \leq \|f\|_p \|p(t,x-\cdot)\|_q = \|f\|_p\|p(t,\cdot)\|_q.$$
Now by scaling, we have for each $t >0$, 
$$ p(t,z) = t^{-d/\alpha}p(1,t^{-d/\alpha}z),$$
and so
$$ \|p(t,\cdot)\|_q = t^{-d/\alpha} t^{d/(\alpha q)} \|p(1,\cdot) \|_q = t^{-d/(\alpha p)}\|p(1,\cdot) \|_q. $$

Thus  
$$ |P_t f(x)|   \leq t^{-d/(\alpha p)}\|p(1,\cdot) \|_q \|f\|_p,$$ 
which implies that 
$$ |R_\lambda f(x) | \leq \|p(1,\cdot)\|_q \| f \|_p \int_0^\infty e^{-\lambda t}t^{-d/(\alpha p)}dt = c_2\|f \|_p,$$
since  $p > \max\{1,d/\alpha\}$.
Then since $c_2$ does not depend on our choice of $x$, \eqref{bdd} follows as well.
  
\qed
\end{proof}

\begin{proposition} \label{p33}
There exists $p_0$ and $c_1$ not depending on $f$ such that for  every $\lambda > 0$, 
$$ \displaystyle \left| \ee \int_0^\infty e^{-\lambda t} f(X_t)\thinspace dt \right| \leq c_1 \|f\|_{p_0}.$$ 
\end{proposition}

\begin{proof}
By \cite{bt}, Corollary 4.5, we have under Assumption \ref{ouch} that
\begin{equation}
\|\calB R_\lambda f\|_p \leq c_2(\zeta + \lambda^{-1}) \|f\|_p, \qquad f\in L^p \cap C_b^2, \quad p \geq 2.
\end{equation}
Choosing $\lambda_0$ sufficiently large, then, we have that  
$$\|\calB R_\lambda f\|_p \leq \frac{1}{2} \|f\|_p, \qquad \lambda > \lambda_0.$$  

In addition, by Lemma \ref{l32}, for $p_0 < \infty$ large enough, $$ \|R_\lambda f\|_{\infty} \leq c_3\|f\|_{p_0}.$$

We now define 
$$ S_\lambda f(x) = \eex \int_0^\infty e^{-\lambda t}f(X_t)dt.$$  
It is well known that for $f \in C_b^2$,  
$$ S_\lambda f = R_\lambda f(x_0) + S_\lambda \calB R_\lambda f.$$
(For a proof of this, see \cite{bc1}, Proposition 6.1.)

Therefore, 
\begin{align*}
\|S_\lambda f \|_\infty& = \displaystyle \left \|R_\lambda  \left(\sum_{i=0}^{\infty} (\calB R_\lambda)^i\right)f \right \|_\infty \\
& \leq c_3 \displaystyle \left \| R_\lambda  \left(\sum_{i=0}^{\infty} (\calB R_\lambda)^i\right)f \right \|_{p_0} \\
 & \leq 2c_3 \|f \|_{p_0}.
\end{align*}

We will show that this implies our result for general $\lambda > 0$.
We observe
\begin{align*}
\ee \int_0^\infty e^{-\lambda t} f(X_t)\thinspace dt &= \ee \sum_{i=0}^\infty \int_i^{i+1} e^{-\lambda t} f(X_t)\thinspace dt \\
&\leq \ee \sum_{i=0}^\infty e^{-\lambda i} \int_i^{i+1} f(X_t)\thinspace dt \\
&\leq  \ee \left[ \sum_{i=0}^\infty e^{-\lambda i} \thinspace \ee^{X_i} \int_0^{1}  f(X_t)\thinspace dt \right] \\
&\leq \left(\sum_{i=1}^\infty e^{-\lambda i}\right)\left[ \sup_{y \in \rr^d} \eey \int_0^1 f(X_t)dt \right] \\
&\leq c_2 \|f\|_{p_0}.
\end{align*}
Here, $c_2$ depends on $\lambda_0$.
\qed
\end{proof}

We note now that in the following arguments, we will always we dealing with points in small balls (or cubes), so in fact our temporary Assumption \ref{ouch} is implied by Assumption \ref{a22}.

We now progress to one of the chief goals of this paper, which is to show that we can expect the processes discussed here to spend some positive amount of time in a set having positive Lebesgue measure.  To do this, we essentially mimic the analogous proof  in the nondivergence case given in  \cite{bass}.  First we show this result in the case where $B$ is almost the entire cube $Q(0,1).$

\begin{proposition} \label{p41}
There exist $c_1$ and $\varepsilon$ such that if $B \subseteq Q(0,1)$, $x \in Q(0,1/2)$, and $|Q(0,1)-B| < \varepsilon$, then  $$\eex \int_0^{\tau_{Q(0,1)}}1_B(X_s)ds \geq c_1.$$
\end{proposition}

\begin{proof}
Let us denote $\tau_{Q(0,1)}$ by $\tau$.  By Lemma \ref{l24}, there exists $c_2$ such that $\eex \tau \geq c_2$, and by Lemma \ref{l25}, we have that $\sup _x \ee^x \tau \leq c_3$, so that $\eex \tau^2 \leq c_4.$

Since 
$$ \eex(\tau - (\tau \wedge t_0)) \leq \eex (\tau ; \tau \geq t_0) \leq \eex \tau^2 / t_0,  $$    
we are able to choose $t_0$ large enough to ensure that $\eex(\tau - (\tau \wedge t_0))  \leq c_2 / 4 $.  Therefore, 
\begin{align*}
\eex \int_0^\tau &1_{(Q(0,1) -B)}(X_s)ds \\
& \leq c_2 /4 + e^{t_0}\eex \int_0^{t_0} e^{-s} 1_{(Q(0,1) -B)}(X_s)ds \\
& \leq c_2 /4 + e^{t_0}\eex \int_0^\infty e^{-s} 1_{(Q(0,1) -B)}(X_s)ds \\
& \leq c_2 /4 + e^{t_0}\eex \int_0^\infty e^{-s\lambda} 1_{(Q(0,1) -B)}(X_s)ds \\ 
& \leq c_2 / 4 + c_5e^{t_0}\varepsilon^{1/p_0}, 
\end{align*}
by Proposition \ref{p33}, with  $p_0$ chosen so as to satisfy this proposition.  Thus if we choose $\varepsilon$ small enough, then $\eex \int_0^\tau 1_{(Q(0,1) -B)}(X_s)ds < c_2 /2$, so this proposition will hold with $c_1 = c_2 / 2$.
\qed
\end{proof}

\begin{lemma} \label{l42}
Suppose $r > 1$ and let $W$ be a cube in $Q(0,1)$.  Let $W^*$ be the cube with the same center as $W$ but side length half as long.  Let $V$ be a subset of $W$ with the property that there exists $\delta$ such that
$$ \eey \int_0^{\tau_W} 1_V(X_s)ds \geq \delta \eey \tau_W, \qquad y \in W^*.$$
Then there exists $\xi(\delta)$ depending on $\delta$ and r such that
$$ \eey \int_0^{\tau_{Q(0,r)}} 1_V(X_s)ds \geq \xi(\delta) \eey \int_0^{\tau_{Q(0,r)}} 1_W(X_s)ds, \qquad y \in Q(0,1). $$
\end{lemma}

\begin{proof}
Let $S$ be the cube in $Q(0,r)$ with the same center as $W$ but side length $r \wedge 2^{1/d}$ as long.  Let $T_1 = \inf \{t: X_t \in W \}$, $U_1 = \inf \{t > T_1 : X_t \notin S \}$, $T_{i+1} = \inf \{t > U_i: X_t \in W \}$, and $U_{i+1} = \inf \{t > T_{i+1} : X_t \notin S \}$.  Then
\begin{align*}
 \eey \int_0^{\tau_{Q(0,r)}}1_W(X_s)ds  &= \sum \eey \Bigl[ \int_{T_i}^{U_i}1_W(X_s)ds ; T_i < \tau_{Q(0,r)}    \Bigr], \\ 
&=  \sum \eey \Bigl[\ee^{X(T_i)}  \int_{0}^{\tau_S}1_W(X_s)ds ; T_i < \tau_{Q(0,r)}    \Bigr],
\end{align*}
and similarly this equation also holds if we replace $W$ by $V$.  Thus we need to show that there exists a $\xi(\delta)$ such that 
$$ \ee^w \int_0^{\tau_S}1_V(X_s)ds \geq \xi(\delta) \ee^w \int_0^{\tau_S} 1_W(X_s)ds, \qquad w \in W. $$

We observe that the proportion of $S$ made up of $W$ is $\frac{1}{2^dr^d} \vee \frac{1}{2^{d+1}}$, a quantity which does not depend the size of $W$, so by Proposition \ref{p26} there exists $c_1$ only depending on $r$ such that
$$ \pp^w(T_{W^*} < \tau_S) \geq c_1, \qquad w \in W. $$

So if $w \in W$, the strong Markov property implies that
\begin{align*}
\ee^w \int_0^{\tau_S} 1_V(X_s)ds &\geq \ee^w \Bigl[  \int_0^{\tau_S} 1_V(X_s)ds ; T_{W^*} < \tau_S \Bigr] \\
&= \ee^w \Bigl[ \ee^{X(T(W^*))}  \int_0^{\tau_S} 1_V(X_s)ds ; T_{W^*} < \tau_S \Bigr] \\
&\geq c_1 \inf_{z \in W^*} \eez \int_0^{\tau_S} 1_V(X_s)ds \\
&\geq c_1 \inf_{z \in W^*} \eez \int_0^{\tau_W} 1_V(X_s)ds. 
\end{align*}

By hypothesis, if $z \in W^*$, 
$$ \eez \int_0 ^{\tau_W}1_V(X_s) ds \geq \delta\eez \tau_W.$$

By Lemma \ref{l24} and scaling, 
$$ \eez \tau_W \geq c_2 \sup_{v \in S} \ee^v \tau_S  \geq  c_2 \ee^w  \int_0^{\tau_S} 1_W(X_s) ds.$$

Taking $\xi(\delta) = c_1c_2\delta$ completes the proof.  Note that the way which $c_1$ and $c_2$ were chosen implies that neither constant can be greater than one, so we have that $\xi(\delta) \leq \delta.$ 
\qed
\end{proof} 

\begin{theorem} \label{t43}
There exists a nondecreasing function $\varphi : (0,1) \to (0,1)$ such that if $x \in Q(0, 1/2)$ and $B \subseteq Q(0,1)$, then
$$ \eex \int_0^{\tau_{Q(0,1)}}1_B(X_s) ds \geq \varphi(|B|). $$ 
\end{theorem}

\begin{proof}
Let
\begin{align*}
\varphi(\varepsilon) = \inf\Bigl\{ \eey \int_0^{\tau_{Q(z_0,R)}}&1_B(X_s)ds \colon z_0 \in \rr^d, R>0,B \subseteq Q(z_0,R), \\
 &|B| \geq \varepsilon|Q(z_0,R)|, y \in Q(z_0,R/2) \Bigr\}.
\end{align*}
By Proposition \ref{p41} and scaling, we obtain that $\varphi(\varepsilon) > 0$ for $\varepsilon$ sufficiently close to 1.  Our goal, then, is to show that $\varphi(\varepsilon) > 0$ for all positive $\varepsilon$.

Let $q_0$ be the infimum of the $\varepsilon$ for which $\varphi(\varepsilon) > 0$.  We will argue by contradiction, and will suppose that $q_0 > 0$.  Since $q_0 < 1$, there exists a $q > q_0$ such that $(q + q^2)/2 <q_0$.  Set $\gamma = (q - q^2)/2$.  Let $\beta$ be a number of the form $2^{-n}$ with
$$ (\gamma \wedge q \wedge(1-q))/32d \leq \beta < (\gamma \wedge q \wedge(1-q))/16d. $$

Since $\xi(\delta) \leq \delta$ and $\varphi$ is an increasing function, there exist $z_0 \in \rr^d$, $R > 0$, $B_1 \subseteq Q(z_0, R)$, and $x \in Q(z_0, R/2)$ such that $q > |B_1|/|Q(z_0,R)| > q - \gamma / 2$ and 
$$ \eex  \int_0^{\tau_{Q(z_0,R)}}1_{B_1}(X_s)ds < \xi(\varphi(q))\varphi(q),$$
where $\xi$ is defined in Lemma \ref{l42}.  Without loss of generality, we can suppose $z_0 = 0$ and $R = 1$, so that
$$\eex  \int_0^{\tau_{Q(0,1)}}1_{B_1}(X_s)ds < \xi(\varphi(q))\varphi(q).$$

Let $B = B_1 \cap Q(0,1-\beta)$.  Then
$$\eex  \int_0^{\tau_{Q(0,1)}}1_{B}(X_s)ds < \xi(\varphi(q))\varphi(q),$$
and by our choice of $\beta$, $q > |B| > q - \gamma$.

As in the Harnack inequality proof given by Krylov and Safonov \cite{ks}, we construct $D$ consisting of the union of cubes $\widehat R_i$, such that $$|D \cup Q(0,1)| \geq \frac{|B|}{q} > \frac{q - \gamma}{q} = \frac{q + 1}{2},$$ and such that $ |B \cup R_i| > q|R_i|$ for all $i$.  We also have that the $R_i$ have pairwise disjoint interiors, where $R_i$ is the cube with the same center as $\widehat R_i$ and one-third the side length.  (For a proof of this, see Chapter 5, Section 7 of \cite{bass}.)  Let $\widetilde D =D \cap Q(0,1)$.  Then we see that
$$|\widetilde D|  \geq (q + 1)/2 > q > q_0,$$
and therefore
$$ \eex \int_0^{\tau_{Q(0,1)}}1_{\widetilde D}(X_s)ds > \varphi(q).$$

Let $V_i = \widehat R_i \cap Q(0,1-\beta)$.  We want to show for each $i$,
\begin{equation} \label{goal}
\eex  \int_0^{\tau_{Q(0,1)}}1_{B \cap R_i}(X_s)ds \geq \xi(\varphi(q)) \eex  \int_0^{\tau_{Q(0,1)}}1_{V_i}(X_s)ds.
\end{equation}

Once we have \eqref{goal}, we sum and obtain 
\begin{align*} 
\eex  \int_0^{\tau_{Q(0,1)}}1_{B}(X_s)ds &\geq \sum_i \int_0^{\tau_{Q(0,1)}}1_{B \cap R_i}(X_s)ds \\
&\geq \xi(\varphi(q))  \sum_i \eex \int_0^{\tau_{Q(0,1)}}1_{V_i}(X_s)ds \\
&\geq \xi(\varphi(q))  \eex \int_0^{\tau_{Q(0,1)}}1_{\widetilde D}(X_s)ds \\
&\geq \xi(\varphi(q))\varphi(q),
\end{align*}
which is our desired contradiction.

We now prove \eqref{goal}.  Fix $i$.  By our definition of $\beta$, if $V_i$ is not empty, then $V_i$ is contained in a cube $W_i$ which is itself a subset of $Q(0,1-\beta)$, such that $|W_i| \leq 3^d|R_i|$.  Let $ R_i^*$ be the cube with the same center as $R_i$ but side length half as long.  By the definition of $\varphi$, 
$$ \eey \int_0^{\tau_{R_i}} 1_{B \cap R_i}(X_s)ds \geq \varphi(q)\eey \tau_{R_i} $$
if $ y \in R_i^*$.  We can now deduce \eqref{goal} from Lemma \ref{l42} and scaling. 
\qed
\end{proof}

\section{Equicontinuity and Approximation}\label{approx}

We recall that the convolution of two functions is defined by $f * g(x) = \int f(y) g(x-y)\thinspace dy$.  Let $\varphi$ be a nonnegative radially symmetric function with compact support such that $\int_{\rr^d}\varphi = 1$ and $\varphi > 0$ on $B(0,r)$ for some $r$.  Let $\varphi_\varepsilon(x) = \varepsilon^{-d}\varphi(x/\varepsilon)$.

\begin{theorem} \label{t61}
Let $\lambda > 0$.  There exist $n^\varepsilon(x,h)$ in $C^\infty(x)$ with the following properties: \\
(i) $n^\varepsilon(x,h)$ satisfies Assumption \ref{a11}. \\
(ii) If 
\begin{equation}\label{e61}
\LL^\varepsilon f(x) = \int _{\rr^d - \{0\}} [f(x+h)-f(x) - \nabla f(x) \cdot h1_{(|h| \leq 1)}] n^\varepsilon(x,h) dh,
\end{equation}
$\ppx_\varepsilon$ is the solution for the martingale problem for $\LL^\varepsilon$ started at $x$, and 
\begin{equation}\label{e62}
S_\lambda^\varepsilon h(x) = \eex_\varepsilon \int_0^\infty e^{-\lambda t} h(X_t) \thinspace dt
\end{equation}
for $h$ bounded, then

\begin{equation}
(S_\lambda^\varepsilon f * \varphi_\varepsilon)(x_0) \to \eex \int_0^\infty e^{-\lambda t} f(X_t) \thinspace dt
\end{equation}
whenever $f$ is continuous.

\end{theorem}

\begin{proof}
Define a measure $\mu$ by 
\begin{equation}\label{mu}
\mu(C) = \ee \int_0^\infty  e^{-\lambda t} 1_C(X_{t}) \thinspace dt.
\end{equation}
We claim that for each $y \in \rr^d$ and $r > 0$, there is some positive probability that $X_t$ starting at $x_0$ enters the ball $B(y,r)$ and stays there a positive length of time.  To see this, let $\psi \colon [0,1] \to \rr^d$ be continuous with $\psi(0) = x_0$ and $\psi(1) = y$, such that $|y - \psi(t)| < r/2$ for all $t \in [1/2,1]$, and then apply the support theorem with $\varepsilon = r/2$.  This result implies that $\mu(B(y,r)) > 0$ for all $y$ and $r$.  We define
\begin{equation}\label{e65}
\displaystyle n^\varepsilon(x,h) = \frac{\int \varphi_\varepsilon(x - y) n(y,h) \mu(dy)}{\int \varphi_\varepsilon(x - y) \mu(dy)} .
\end{equation}
It follows from our assumptions on  $\varphi$ that the denominator is nonzero.  It is clear that (i) holds.

Suppose $u$ is a bounded $C^2$ function.  By Ito's product formula, 
\begin{align*}
e^{-\lambda t}u(X_t) &=  u(X_0) - \int_0^t u(X_{s-})\lambda e^{-\lambda s}\thinspace ds + \int_0^t e^
{-\lambda s} d[u(X_s)]_s \\
&= u(X_0) -  \int_0^t u(X_{s-})\lambda e^{-\lambda s}\thinspace ds + \text{martingale} + \int_0^t e^
{-\lambda s} \LL u(X_{s-}) \thinspace ds.
\end{align*}

Suppose $X_0 = x_0$.  We take expectations and let $t \to \infty$, to obtain
\begin{equation}\label{e66}
u(x_0) = \ee \int_0^\infty e^{-\lambda s} (\lambda u - \LL u)(X_{s-})\thinspace ds = \int (\lambda u - \LL u)(x) \thinspace \mu(dx).
\end{equation}

We now let $v$ be a bounded, $C^2$ function, and we apply \eqref{e66} to $u = v * \varphi_\varepsilon$.  On the left-hand side we have $\int v(x_0 -y)\varphi_\varepsilon(y) \thinspace dy$. We now observe that
\begin{align}\label{e67}
\LL(u)(z) &= \int [u(z+h)-u(z) - \nabla u(z) \cdot h1_{(|h| \leq 1)}] n(z,h) dh  \\
&= \int [u(z+h)-u(z)]n(z,h)dh - \int_{|h| \leq 1}[ \nabla (v * \varphi_\varepsilon)(z) \cdot h ]n(z,h) dh \notag \\
&= \int [u(z+h)-u(z)]n(z,h)dh - \int_{|h| \leq 1}\sum_{i=1}^d ( \partial_iv * \varphi_\varepsilon)(z)p_i(h)n(z,h)dh \notag \\
&= \int [u(z+h)-u(z)]n(z,h)dh - \int_{|h| \leq 1}\sum_{i=1}^d \int \partial_iv(x)\varphi_\varepsilon(x-z)p_i(h)n(x,z)dz\thinspace dh.\notag 
\end{align}
where $p_i$ is the projection onto the $i$th coordinate.

Furthermore, by definition, 
\begin{align*}
&\int [u(z+h)-u(z)]n(z,h)dh \\ &= \int [(v * \varphi_\varepsilon)(z+h)-(v * \varphi_\varepsilon)(z)]n(z,h)dh \\
&= \int \int v(x)\varphi_\varepsilon(x-(z+h))n(z,h)dy\thinspace dh - \int \int v(x)\varphi_\varepsilon(x-z) n(z,h)dy\thinspace dh.
\end{align*}

However, by \eqref{e65}, 
\begin{equation} \label{e68}
\int \varphi_\varepsilon(x-z)n(z,h)\mu(dy) = n^\varepsilon(x,h)\int \varphi_\varepsilon(x-y) \mu(dy).
\end{equation}

Thus, combining \eqref{e66}, \eqref{e67}. and \eqref{e68}, 
\begin{align}
\int v(x_0 - y)\varphi_{\varepsilon}(y)dy
 &= \int [ \lambda(v * \varphi_\varepsilon) - \LL(v * \varphi_\varepsilon)](x) \mu(dx) \label{e69} \\
&= \int \int (\lambda - \LL^\varepsilon)v(x)\varphi_\varepsilon(x-y) \mu(dy) dx. \notag 
\end{align}

Suppose that $f$ is smooth, and let $v(x) = S_\lambda^\varepsilon f(x)$.  It follows from results of \cite{bass2}
that $v$ is $C^2$, and from Proposition \ref{p33}, we have that $v$ is bounded.  We further claim that $(\lambda - \LL^\varepsilon)v  = f$.  To see this, let $P^\varepsilon_t$ be the  transition semigroup associated to $X$. 
 We observe that we can write

$$ S_\lambda^\varepsilon f(x) = \int_0^\infty e^{-\lambda t}P^\varepsilon_t f(x) \thinspace dt,$$
so that 
\begin{align*}
\LL^\varepsilon(G ^\lambda_\varepsilon f(x)) &= \displaystyle \lim_{h\to 0} \frac{P^\varepsilon_hG ^\lambda_\varepsilon f(x) - G ^\lambda_\varepsilon f(x)}{h} \\
 &= \lim_{h\to 0} \frac{ \int_0^\infty e^{-\lambda t} P^\varepsilon_{t+h}\thinspace dt - \int_0^\infty e^{-\lambda t}P^\varepsilon_t f(x) \thinspace dt}{h}\\
&=  \lambda - f(x). 
\end{align*}

Therefore, we can substitute in \eqref{e69}, to obtain
\begin{align}
\int S_\lambda^\varepsilon f(x_0 - y) \varphi_\varepsilon (y)\thinspace dy &= \int \int f(x) \varphi_\varepsilon(x-y) \mu(dy) \thinspace dx \label{e610} \\
&= \int f * \varphi_\varepsilon(y) \mu(dy). \notag
\end{align}

By a limit argument, we have \eqref{e610} when $f$ is continuous.  Since $f$ is continuous, $f * \varphi_\varepsilon$ is bounded and converges uniformly to $f$.  Therefore, 
$$ \int f * \varphi_\varepsilon (y) \mu(dy) \to \int f(y) \mu(dy) = \ee \int_0^\infty e^{-\lambda t} f(X_t) dt.$$
\qed
\end{proof}

We now use Proposition \ref{p33} to extend this result further.

\begin{theorem} \label{c63}
Under the assumptions of Theorem \ref{t61}, 
$$(S_\lambda^\varepsilon f * \varphi_\varepsilon)(x_0) \to \ee \int_0^\infty e^{-\lambda t} f(X_t) \thinspace dt, $$
if $f$ is bounded.
\end{theorem}

\begin{proof}
In the proof of Theorem \ref{t61}, we have that \eqref{e610} holds when $f$ is continuous, and by a limit argument, we have that \eqref{e610} holds for $f$ bounded.  Therefore, it suffices to show that the right-hand side of \eqref{e610} converges to $\int f(y) \mu(dy)$.  It is known that since $f$ is bounded, $f * \varphi_\varepsilon$ converges to $f$ almost everywhere and boundedly.  By Proposition \ref{p33} and \eqref{mu}, $\mu$ is absolutely continuous with respect to Lebesgue measure.  Then by dominated convergence, 
\begin{align*}
\int f * \varphi_\varepsilon(y) \mu(dy) &=  \int f * \varphi_\varepsilon(y) (d\mu/dy) dy \\
&\to \int f(y) (d\mu/dy) dy = \int f(y) \mu(dy).
\end{align*}
\qed
\end{proof}

\begin{theorem}\label{t82}
Let $\lambda > 0$, and let $\pp$ be a solution to the martingale problem for $\LL$ started at $x$.  There exist $n^\varepsilon(x,h)$ which are smooth with respect to $x$, such that if $\LL^\varepsilon$ is defined by \eqref{e61}, and $S_\lambda^\varepsilon$ is defined by \eqref{e62}, then
$$ S_\lambda^\varepsilon f(x) \to \ee \int_0^\infty e^{-\lambda t} f(X_t) dt $$
when $f$ is bounded.
\end{theorem}

\begin{proof}
By \cite{bl} Theorem 4.3 , $S_\lambda^\varepsilon f$ is equicontinuous in $\varepsilon$.  Since in Theorem \ref{c63}, $\varphi_\varepsilon$ has compact support, 
\begin{align*}
&|(S_\lambda^\varepsilon f * \varphi_\varepsilon)(x_0) - S_\lambda^\varepsilon f(x_0)| \\
& \leq \int |S_\lambda^\varepsilon f(x_0 - \varepsilon y) - S_\lambda^\varepsilon f(x_0)|\varphi(y)dy \to 0
\end{align*} 
as $\varepsilon \to 0$.  The result now follows from Theorem \ref{c63}.
\qed
\end{proof}

Our lower bound on the occupation times of sets can now be used to show that the $n^\varepsilon(x,h)$ defined above converge almost everywhere to $n(x,h)$.  

\begin{proposition}\label{c86}
If $|B| > 0$, then $\eex \int_0^\infty e^{-\lambda t} 1_B(X_t)dt > 0.$
\end{proposition}

\begin{proof}
Let $Q$ be a unit cube such that $|Q \cap B| > 0$.  Let $Q^*$ be the cube having the same center as $Q$ does, but with side length half as long.  By the strong Markov property and the support theorem, there exists $c_1$ such that
$$\displaystyle \eex \int_0^\infty e^{-\lambda t} 1_B(X_t)dt \geq \inf_{y \in Q^*} \eey \int_0^\infty e^{-\lambda t} 1_{Q \cap B}(X_t)dt.$$
By \cite{bl}, Lemma 3.3, there exists $c_2$ such that if $y \in Q^*$, 
\begin{align*}
\eey(\tau_Q - (\tau_Q \wedge t)) &= \eey (\ee^{X_t}\tau_Q; t < \tau_Q) \\
&\leq c_2 \pp^y(t < \tau_Q) \leq c_2 \eex \tau_Q^2/t^2.
\end{align*}
After again applying \cite{bl}, Lemma 3.3, we can take $t_0$ large enough so that 
$$\sup_{y \in Q^*} \eey(\tau_Q - (\tau_Q \wedge t)) \leq \psi(|Q \cap B|) / 2, $$
where $\psi$ is the function obtained in Theorem \ref{t43}.  Therefore, we have
\begin{align*}
\eey \int_0^\infty e^{-\lambda t} 1_{Q \cap B}(X_t)dt &\geq \eey \int_0^{\tau_Q \wedge t_0} e^{-\lambda t} 1_{Q \cap B}(X_t)dt\\
&\geq e^{-\lambda t_0} \eey \int_0^{\tau_Q \wedge t_0}  1_{Q \cap B}(X_t)dt\\
&\geq e^{-\lambda t_0}\left[ \eey \int_0^{\tau_Q}1_{Q \cap B}(X_t)dt - \eey (\tau_Q - (\tau_Q \wedge t_0))\right]\\
&\geq e^{-\lambda t_0}(\psi(|Q \cap B|) - \psi(|Q \cap B|)/2) > 0.
\end{align*}
by applying Theorem \ref{t43}.
\qed
\end{proof}

\begin{proposition}\label{p87}
If $|C| > 0$, then $\ee \int_0^\infty e^{-\lambda t}1_C(X_t)dt > 0$.
\end{proposition}
 
\begin{proof}
This follows immediately from Proposition \ref{c86} and Theorem \ref{t82} with $f = 1_C$.\qed
\end{proof}

\begin{theorem}
Let $n^\varepsilon(x,h)$ be defined by \eqref{e65}.  Then for each fixed $h$, $n^\varepsilon(x,h) \to n(x,h)$ almost everywhere.
\end{theorem}

\begin{proof}
By Proposition \ref{p33}, 
$$\left| \int f(y) \mu(dy) \right| = \left| \ee \int_0^\infty e^{-\lambda t} f(X_t) dt \right| \leq c_1 \|f\|_{p_0}.$$
Therefore, $\mu(dy)$ has a density $m(y) \thinspace dy$, and by a duality argument, $m \in L^{p_0/(p_0-1)}$.  Define $C = \{y\colon m(y) = 0\}$.  We observe that, 
$$\ee \int_0^\infty e^{-\lambda t}1_C(X_t)dt = \int 1_C(y)\mu(dy) = \int1_C(y) m(y) dy = 0,$$
so by Proposition \ref{p87}, $|C| = 0$.

Now, 
\begin{align*}
\int \varphi_\varepsilon(x - y) n(y,h) \mu(dy) &= \int \varphi_\varepsilon(x - y) n(y,h) m(y) dy \\
&\to n(x,h)m(x)
\end{align*}
for almost every $x$, since $\varphi_\varepsilon$ is an approximation to the identity, $n(x,h)$ is bounded for any fixed $h$, and $m \in L^{p_0/(p_0-1)}$.  Similarly, 
$$\int \varphi_\varepsilon(x - y) \mu(dy) = \int \varphi_\varepsilon(x - y) m(y) dy \to m(x)$$
for almost every $x$.  Since $m > 0$ almost everywhere, the ratio, which is $n^\varepsilon(x,h)$, converges to $n(x,h)$ almost everywhere.
\qed
\end{proof}



\begin{thebibliography}{51}
%
%
\bibitem{bbck}
Barlow, M.T., Bass, R.F., Chen, Z.-Q., Kassmann, M.: Non-local Dirichlet forms and symmetric jump processes. Trans. Amer. Math. Soc. {\bf 361}, 1963-1999 (2009)

\bibitem{bbg}
Barlow, M.T., Bass, R.F., Gui, C.: The Liouville property and a conjecture of De Giorgi. Comm. Pure Appl. Math. {\bf 53}, 1007-1038 (2000)

\bibitem{bass}
Bass, R.F.: Diffusions and Elliptic Operators. Springer, New York (1997)

\bibitem{bass2}
Bass, R.F.: Regularity results for stable-like operators. J. Functional Anal. {\bf 257}, 2693-2722 (2009)

\bibitem{ba}
Bass, R.F.: Uniqueness in law for pure jump Markov processes. Probab. Th. rel. Fields {\bf 79}, 271-287 (1988)

\bibitem{bc1}
Bass, R.F., Chen, Z.-Q.: Systems of equations driven by stable processes. Prob. Theory rel. Fields {\bf 134}, 175-214 (2006)

\bibitem{bc2}
Bass, R.F., Chen, Z.-Q.: Regularity of harmonic functions for a class of singular stable-like processes. Math. Z. {\bf 266}, 489-503 (2010)  

\bibitem{bl}
Bass, R.F., Levin, D.A.: Harnack inequalities for jump processes. Potential Anal. {\bf 17}, 375-388 (2009)

\bibitem{bk}
Barlow, M.T., Kassmann, M.: Harnack inequalities for non-local operators of variable order. Trans. Amer. Math. Soc. {\bf 357}, 837-850 (2005)

\bibitem{bt}
Bass, R.F., Tang, H.: The martingale problem for a class of stable-like processes. Stoch. Proc. \& their Applic. {\bf 119}, 1114-1167 (2009)

\bibitem{k}
Komatsu, T.: On the martingale problem for generators of stable processes with perturbations. Osaka J. Math. {\bf 21}, 113-132 (1984)

\bibitem{ks}
Krylov, N.V., Safonov, M.V.: A property of the solutions of parabolic equations with measurable coefficients. Izv. Akad. Nauk SSSR Ser. Mat. {\bf 44}, 161-175 and 239 (1980)

\bibitem{sv}
Song, R., Vondra\u{c}ek, Z: Harnack inequality for some classes of Markov processes. Math. Z. {\bf 246}, 177-202 (2004)
\end{thebibliography}


\end{document}